\documentclass[10pt,reqno,a4paper]{amsart}

\usepackage{amsfonts, amsmath, amsthm, amssymb}
\usepackage[english]{babel}
\usepackage[mathcal]{eucal}
\usepackage{graphicx}
\usepackage[all]{xy}

\allowdisplaybreaks

\title[]{A note on probability metrics in a categorical setting}
\author{Ben Berckmoes and Bob Lowen}
\date{}

\begin{document}

\dedicatory{Dedicated to the memory of Horst Herrlich, a great mathematician and a great friend.}

\maketitle

\newtheorem{pro}{Proposition}[section]
\newtheorem{lem}[pro]{Lemma}
\newtheorem{thm}[pro]{Theorem}
\newtheorem{de}[pro]{Definition}
\newtheorem{co}[pro]{Comment}
\newtheorem{no}[pro]{Notation}
\newtheorem{vb}[pro]{Example}
\newtheorem{vbn}[pro]{Examples}
\newtheorem{gev}[pro]{Corollary}
\newtheorem{rem}[pro]{Remark}
\newtheorem{ass}[pro]{Assumption}

\begin{abstract}
Probability metrics constitute an important tool in probability theory and statistics \cite{DKS91}, \cite{R91}, \cite{Z83} as they are specific metrics on spaces of random variables which, by satisfying an extra condition, concord well with the randomness structure. But probability metrics suffer from the same instability under constructions as metrics. In \cite{L15}, as well as in former and related work which can be found in the references of \cite{L15}, a comprehensive setting was developed to deal with this. It is the purpose of this note to point out that these ideas can also be applied to probability metrics thus embedding them in a natural categorical framework, showing that certain constructions performed in the setting of probability theory are in fact categorical in nature. This allows us to deduce various separate results in the literature from a unified approach.
\end{abstract}

\section{Introduction}

We define a basic concrete category $\mathsf{Pol_u}$ (see section 3) which is a natural ``habitat" for the notion of probability metric  \cite{DKS91}, \cite{R91}, \cite{Z83}. This has several advantages. In the first place, by extending the notion of probability metric to a ``uniform approach theoretic" counterpart we make it possible to perform constructions (see the Ky Fan and Prokhorov probability uniform gauge structures defined in Examples \ref{vbn:PGs}) which in the metric realm are unavailable. In the second place we are able to describe a much used association of minimal probability metric to a given probability metric, as a simple reflection in our basic category. Consequently, by also describing the underlying approach structures via their limit operators we are able to prove a general result from which four different results in the literature (all with different proofs) can be deduced (see Theorem \ref{thm:LimOpMet}, Corollary \ref{gev:GenConv} and the comments thereafter). Finally also certain relations between the approach structures of convergence in probability and the weak approach structure can be deduced from a general result (see Theorem \ref{thm:MinLimOp} and Corollary \ref{last}).

\section{Probability uniform gauges}

Recall that a Polish space is a completely metrizable and separable topological space. For a Polish space $X$, we denote its Borel $\sigma$-field by $\mathcal{B}_X$, and the collection of probability measures on $\mathcal{B}_X$ by $\mathcal{P}(X)$. Furthermore, for a probability space $(\Omega,\mathcal{F},\mathbb{P})$, we denote the set of $X$-valued random variables on $\Omega$ by $\mathcal{R}(\Omega, X)$. Finally, for $\xi \in \mathcal{R}(X)$, we let $\mathcal{L}(\xi)$ be the image measure of $\xi$, that is, $\mathcal{L}(\xi)[A] = \mathbb{P}[\xi \in A]$ for all $A \in \mathcal{B}_X$. 

\label{ass:RichPS}
Throughout, we fix a single probability space $(\Omega,\mathcal{F},\mathbb{P})$ with the property that for each Polish space $X$ and each probability measure $P \in \mathcal{P}(X)$ there exists a random variable $\xi \in \mathcal{R}(X)$ such that $P = \mathcal{L}(\xi)$. For the existence of such a probability space, we refer the reader to \cite{R91}. Therefore, henceforth, we will eliminate all reference to $\Omega$ from our notations, and simply write $\mathcal{R}(X)$ instead of $\mathcal{R}(\Omega, X)$.

Let $X$ be a Polish space. A map 
\begin{equation*}
d: \mathcal{R}(X) \times \mathcal{R}(X) \rightarrow [0,\infty] : (\xi,\eta) \mapsto d(\xi,\eta) 
\end{equation*}
is called a {\em probability metric} iff it is a metric (neither necessarily with underlying Hausdorff topology nor necessarily finite)
%\begin{itemize}
%\item[(PM1)] $\forall \xi \in \mathcal{R}(X) : d(\xi,\xi) = 0$, 
%\item[(PM2)] $\forall \xi, \eta \in\mathcal{R}(X) : d(\xi,\eta) = d(\eta,\xi)$,
%\item[(PM3)] $\forall \xi,\eta, \zeta \in \mathcal{R}(X) : d(\xi,\eta) \leq d(\xi,\zeta) + d(\zeta,\eta)$, 
%\end{itemize}
which satisfies the additional property
\begin{itemize}
\item[(PM)] $\forall \xi, \eta, \xi^\prime, \eta^\prime \in \mathcal{R}(X) : \mathcal{L}(\xi,\eta) = \mathcal{L}(\xi^\prime,\eta^\prime) \Rightarrow d(\xi,\eta) = d(\xi^\prime,\eta^\prime)$.
\end{itemize}

%\begin{rem}
\label{rem:PQMreg}
If $d$ is a probability metric on $\mathcal{R}(X)$ then it follows immediately from {\upshape(PM)} that 
\begin{equation*}
(\mathbb{P}[\xi \neq \xi^\prime] = 0 \textrm{ and } \mathbb{P}[\eta \neq \eta^\prime] = 0) \Rightarrow d(\xi,\eta) = d(\xi^\prime,\eta^\prime)
\end{equation*}
for all $\xi,\eta,\xi^\prime,\eta^\prime \in \mathcal{R}(X)$.
%\end{rem}

\begin{vbn}\label{vbn:PMs}
Let $(X,d)$ be a complete and separable metric space. The following metrics are all examples of well-known probability metrics on $\mathcal{R}(X)$, see e.g. \cite{R91} and \cite{Z83}.
\begin{enumerate}
\item{} The {\em Ky-Fan metric} with parameter $\lambda \in \mathbb{R}^+_0$:
\begin{equation*}
K_\lambda(\xi,\eta) = \inf \{\epsilon > 0 \mid \mathbb{P}[d(\xi,\eta) \geq \lambda \epsilon] < \epsilon\},
\end{equation*}
\item{} for $p \in \mathbb{R}^+_0$, the {\em $L^p$ metric}:
\begin{equation*}
d_p(\xi,\eta) = \left(\int d^p(\xi,\eta) d\mathbb{P}\right)^{1/p},
\end{equation*}
\item{} the {\em $L^{\infty}$ metric}:
\begin{equation*}
d_{\infty}(\xi,\eta) = \inf \{\epsilon > 0 \mid \mathbb{P}[d(\xi,\eta) \geq \epsilon] = 0\},
\end{equation*} 
\item the {\em indicator metric}:
\begin{equation*}
d_i(\xi,\eta) = \mathbb{P}[d(\xi,\eta) > 0],
\end{equation*}
\item{} the {\em Prokhorov metric} with parameter $\lambda \in \mathbb{R}^+_0$:
\begin{equation*}
\rho_{\lambda}(\xi,\eta) = \inf \left\{\epsilon > 0 \mid \forall A \in \mathcal{B}_X : \mathcal{L}(\xi)[A] \leq \mathcal{L}(\eta)\left[A^{(\lambda \epsilon)}\right] + \epsilon\right\},
\end{equation*}
where $A^{(\lambda \epsilon)} = \{x \in X \mid \inf_{a \in A} d(x,a) \leq \lambda \epsilon\}$,
%\item{} for $p \in \mathbb{R}^+_0$, the {\em Wasserstein metric} of order $p$:
%\begin{equation*}
%W_p(\xi,\eta) = \inf d_p(\xi^\prime,\eta^\prime),
%\end{equation*} 
%the infimum being taken over all $\xi^\prime,\eta^\prime \in \mathcal{R}(X)$ for which $\mathcal{L}(\xi^\prime) = \mathcal{L}(\xi)$ and $\mathcal{L}(\eta^\prime) = \mathcal{L}(\eta)$,
%\item{} the {\em Wasserstein metric} of order $\infty$:
%\begin{equation*}
%W_\infty(\xi,\eta) = \inf d_\infty(\xi^\prime,\eta^\prime),
%\end{equation*} 
%the infimum being taken over all $\xi^\prime,\eta^\prime \in \mathcal{R}(X)$ for which $\mathcal{L}(\xi^\prime) = \mathcal{L}(\xi)$ and $\mathcal{L}(\eta^\prime) = \mathcal{L}(\eta)$,
\item{} the {\em total variation metric}:
\begin{equation*}
d_{TV}(\xi,\eta) = \sup_{A \in \mathcal{B}_X} \left|\mathcal{L}(\xi)[A] - \mathcal{L}(\eta)[A]\right|.
\end{equation*}
\end{enumerate}
\end{vbn}

In the setting of approach theory, the main way to go from a metric structure to a uniform gauge structure (see \cite{L15}) is precisely via so-called uniform gauges. A uniform gauge $\mathcal{G}$ is a collection of metrics satisfying the following stability property: any quasi-metric $d$ which is \emph{dominated by} $\mathcal{G}$, i.e. for which for all $\varepsilon >0$ and $\omega <\infty$ there exists a $d^{\varepsilon ,\omega }\in \mathcal{G}$
such that
$$
d \wedge \omega \leq
d^{\varepsilon ,\omega } +\varepsilon,
$$
already belongs to $\mathcal{G}$.

We will now call a uniform gauge $\mathcal{G}$ on $\mathcal{R}(X)$ a {\em probability uniform gauge} if it has a basis consisting of probability metrics.

%\begin{rem}
If $\mathcal{G}$ is a probability uniform gauge on $\mathcal{R}(X)$ and $\lambda_{\mathcal{G}}$ the limit operator associated with its underlying approach structure, then it follows easily from the foregoing that 
\begin{equation*}
(\mathbb{P}[\xi \neq \xi^\prime] = 0 \textrm{ and } \forall n : \mathbb{P}[\xi_n \neq \xi_n^\prime] = 0) \Rightarrow \lambda_\mathcal{G}(\xi_n \rightarrow \xi) = \lambda_{\mathcal{G}}(\xi_n^\prime \rightarrow \xi^\prime)
\end{equation*}
for all $\xi, \xi^\prime \in \mathcal{R}(X)$ and all $(\xi_n)_n, (\xi_n^\prime)_n$ in $\mathcal{R}(X)$.
%\end{rem}

\begin{vbn}\label{vbn:PGs}
The two indexed families of probability metrics in Examples \ref{vbn:PMs} give rise to two interesting probability uniform gauges.
\begin{enumerate}
\item{} The parametrized Ky-Fan metrics $\{K_\lambda \mid \lambda \in \mathbb{R}^+_0\}$ are a basis for a unique probability uniform gauge $\mathcal{G}_K$ on $\mathcal{R}(X)$, which we call the {\em Ky-Fan probability uniform gauge}. In \cite{L15} it is shown that the underlying approach structure of $\mathcal{G}_K$ is the approach structure of convergence in probability $\mathcal{A}_{\mathbb{P}}$.

\item{} The parametrized Prokhorov metrics $\{\rho_\lambda \mid \lambda \in \mathbb{R}^+_0\}$ are a basis for a unique probability uniform gauge $\mathcal{G}_P$ on $\mathcal{R}(X)$, which we call the {\em Prokhorov probability uniform gauge}. In \cite{B14} it is shown that the underlying approach structure of $\mathcal{G}_P$ is the weak approach structure $\mathcal{A}_w$.
\end{enumerate}
\end{vbn}

A probability metric $d$ on $\mathcal{R}(X)$ is said to be {\em simple} if
\begin{equation*}
\mathcal{L}(\xi) =  \mathcal{L}(\xi^\prime) \Rightarrow d(\xi,\xi^\prime) = 0\label{eq:ConSimple}
\end{equation*}
for all $\xi,\xi^\prime \in \mathcal{R}(X)$.

%\begin{rem}
\label{rem:SQ}
It is easily seen that a probability metric $d$ on $\mathcal{R}(X)$ is simple if and only if
\begin{equation*}
(\mathcal{L}(\xi) = \mathcal{L}(\xi^\prime) \textrm{ and } \mathcal{L}(\eta) = \mathcal{L}(\eta^\prime) \Rightarrow d(\xi,\eta) = d(\xi^\prime,\eta^\prime) 
\end{equation*}
for all $\xi,\xi^\prime,\eta,\eta^\prime \in \mathcal{R}(X)$.
%\end{rem}

\begin{vbn}
Among the probability metrics defined in Examples \ref{vbn:PMs}, $\rho_\lambda$ and $d_{TV}$ are the simple ones.
\end{vbn}

A probability uniform gauge on $\mathcal{R}(X)$ is called {\em simple} iff it has a basis consisting of simple probability metrics.

%\begin{rem}
If $\mathcal{G}$ is a simple probability uniform gauge on $\mathcal{R}(X)$ and $\lambda_{\mathcal{G}}$ the limit operator associated with its underlying approach structure, then it follows easily from the definitions that
\begin{equation*}
(\forall n : \mathcal{L}(\xi) = \mathcal{L}(\xi_n)) \Rightarrow \lambda_\mathcal{G}(\xi_n \rightarrow \xi) = 0
\end{equation*}
for all $\xi \in \mathcal{R}(X)$ and all $(\xi_n)_n$ in $\mathcal{R}(X)$, and that
\begin{equation*}
(\mathcal{L}(\xi) = \mathcal{L}(\xi^\prime) \textrm{ and } \forall n : \mathcal{L}(\xi_n) = \mathcal{L}(\xi_n^\prime)) \Rightarrow \lambda_\mathcal{G}(\xi_n \rightarrow \xi) = \lambda_{\mathcal{G}}(\xi_n^\prime \rightarrow \xi^\prime)
\end{equation*}
for all $\xi, \xi^\prime \in \mathcal{R}(X)$ and all $(\xi_n)_n, (\xi_n^\prime)_n$ in $\mathcal{R}(X)$.
%\end{rem}

\begin{vb}
Among the probability uniform gauges defined in Examples \ref{vbn:PGs}, $\mathcal{G}_P$ is the simple one.
\end{vb}

\section{The categorical setup}

We now define the following basic category.
\vspace{5pt}

{\em{Objects}}: an object is a pair $(X,\mathcal{G})$ which consists of a Polish space $X$ and a probability uniform gauge $\mathcal{G}$ on $\mathcal{R}(X)$.  Such a space is called a {\em probability uniform gauge space}.

{\em{Morphisms}}: given probability uniform gauge spaces $(X,\mathcal{G}_X)$ and $(Y,\mathcal{G}_Y)$, a morphism is a map $f : X \rightarrow Y$ with the following properties:

\begin{enumerate}
	\item The map $f : (X,\mathcal{B}_X) \rightarrow (Y,\mathcal{B}_Y)$ is $\mathcal{B}_X$-$\mathcal{B}_Y$-measurable.
	\item The map $\overline{f} : (\mathcal{R}(X),\mathcal{G}_X) \rightarrow (\mathcal{R}(\Omega,Y),\mathcal{G}_Y) : \xi \mapsto f \circ \xi$ is a $\mathcal{G}_X$-$\mathcal{G}_Y$-contraction.
\end{enumerate}
These maps will be referred to as  {\em random contractions}.

We denote the category with probability uniform gauge spaces as objects and random contractions as morphisms by $\mathsf{Pol_{u}}$. 

We call a probability uniform gauge space $(X,\mathcal{G}_X)$ {\em simple} iff the probability uniform gauge $\mathcal{G}_X$ is simple. The subcategory of $\mathsf{Pol_{u}}$ consisting of the simple probability uniform gauge spaces is denoted by $\mathsf{sPol_{u}}$.

We will call a probability uniform gauge space where the uniform gauge is generated by a single metric, simply a {\em probability metric space}.

The subcategory of $\mathsf{Pol_{u}}$ consisting of probability metric  spaces will be denoted by $\mathsf{Pol_m}$ and the subcategory thereof consisting of simple probability metric spaces as $\mathsf{sPol_m}$.

Let $d$ be a probability metric on $\mathcal{R}(X)$. Define the map
\begin{equation*}
\widehat{d} : \mathcal{R}(X) \times \mathcal{R}(X) \rightarrow [0,\infty]
\end{equation*}
by setting
\begin{equation}
\widehat{d}(\xi,\eta) = \inf \{d(\xi^\prime,\eta^\prime) | \xi^\prime, \eta^\prime \in \mathcal{R}(X), \mathcal{L}(\xi^\prime) = \mathcal{L}(\xi), \mathcal{L}(\eta^\prime) = \mathcal{L}(\eta)\}.\label{eq:defMinQm}
\end{equation}

Proposition \ref{thm:dhatisMetric} belongs to the folklore of the theory of probability metrics, see e.g.  \cite{R91} and \cite{Z83}. For the sake of self-containedness, we include a proof based on Lemma \ref{lem:PastingLemma}, which is known as the Gluing Lemma.

Let $J \subset K$ be countable sets, $\{X_k \mid k \in K\}$ a collection of Polish spaces, and $\pi \in \mathcal{P}(\Pi_{k \in K} X_k)$. Then we write the image measure of $\pi$ under the projection $\Pi_{k \in K} X_k \rightarrow \Pi_{j \in J} X_j : (x_k)_{k \in K} \mapsto (x_j)_{j \in J}$ as $\pi^{J}$.

\begin{lem}[Gluing Lemma]\label{lem:PastingLemma}
Let $X_0,X_1,X_2$ be Polish spaces, and $\pi_1 \in \mathcal{P}(X_0 \times X_1)$ and $\pi_2 \in \mathcal{P}(X_1  \times X_2)$ such that $\pi_1^{\{1\}} = \pi_2^{\{1\}}$. Then there exists a probability measure $\pi \in \mathcal{P}(X_0 \times X_1 \times X_2)$ such that $\pi^{\{0,1\}} = \pi_1$ and $\pi^{\{1,2\}} = \pi_2$. 
\end{lem}

\begin{proof} See  \cite{V03}, p 208.
\end{proof}

\begin{pro}\label{thm:dhatisMetric}
The map $
\widehat{d} : \mathcal{R}(X) \times \mathcal{R}(X) \rightarrow [0,\infty]
$
 is a simple probability metric.
\end{pro}

\begin{proof}
We prove the only non-trivial assertion, namely that $\widehat{d}$ satisfies the triangle inequality. Let $\xi, \eta, \zeta \in \mathcal{R}(X)$, $\epsilon > 0$, and $\xi^\prime,\zeta^\prime \in \mathcal{R}(X)$ be such that $\mathcal{L}(\xi^\prime) = \mathcal{L}(\xi)$, $\mathcal{L}(\eta^\prime) = \mathcal{L}(\eta)$, and
\begin{equation*}
d(\xi^\prime,\zeta^\prime) \leq \widehat{d}(\xi,\zeta) + \epsilon/2,
\end{equation*}
and let $\zeta^{\prime \prime}, \eta^{\prime \prime} \in \mathcal{R}(X)$ be such that $\mathcal{L}(\zeta^{\prime \prime}) = \mathcal{L}(\zeta)$, $\mathcal{L}(\eta^{\prime \prime}) = \mathcal{L}(\eta)$, and
\begin{equation*}
d(\zeta^{\prime \prime}, \eta^{\prime \prime}) \leq \widehat{d}(\zeta,\eta) + \epsilon/2.
\end{equation*}  
Put $X_0 = X_1 = X_2 = X$. Then, by the gluing lemma, we are allowed to fix $\pi \in \mathcal{P}(X_0 \times X_1 \times X_2)$ for which $\pi^{\{0,1\}} = \mathcal{L}(\xi^\prime,\zeta^\prime)$ and $\pi^{\{1,2\}} = \mathcal{L}(\zeta^{\prime \prime},\eta^{\prime \prime})$. Furthermore, by our assumption on $\Omega$, we find random variables $\xi_0, \eta_0, \zeta_0 \in \mathcal{R}(X)$ such that $\mathcal{L}(\xi_0,\eta_0,\zeta_0) = \pi$. But then
\begin{eqnarray*}
\widehat{d}(\xi,\eta) &\leq& d(\xi_0,\eta_0)\\
 &\leq& d(\xi_0,\zeta_0) + d(\zeta_0,\eta_0)\\ 
&=& d(\xi^\prime,\zeta^\prime) + d(\zeta^{\prime \prime}, \eta^{\prime \prime})\\ 
&\leq& \widehat{d}(\xi,\zeta) + \widehat{d}(\zeta,\eta) + \epsilon,
\end{eqnarray*}
which by the arbitrariness of $\epsilon$ finishes the proof.
\end{proof}

The probability metric $\widehat{d}$ is called the {\em minimal probability metric} associated with $d$.

\begin{vbn}\label{vbn:MinPMs}
For $p \in \mathbb{R}^+_0$, the {\em Wasserstein metric} of order $p$ is defined as:
\begin{equation*}
W_p(\xi,\eta) = \inf d_p(\xi^\prime,\eta^\prime),
\end{equation*} 
the infimum being taken over all $\xi^\prime,\eta^\prime \in \mathcal{R}(X)$ for which $\mathcal{L}(\xi^\prime) = \mathcal{L}(\xi)$ and $\mathcal{L}(\eta^\prime) = \mathcal{L}(\eta)$,
and analogously the {\em Wasserstein metric} of order $\infty$ is defined as:
\begin{equation*}
W_\infty(\xi,\eta) = \inf d_\infty(\xi^\prime,\eta^\prime),
\end{equation*} 
the infimum being taken over all $\xi^\prime,\eta^\prime \in \mathcal{R}(X)$ for which $\mathcal{L}(\xi^\prime) = \mathcal{L}(\xi)$ and $\mathcal{L}(\eta^\prime) = \mathcal{L}(\eta)$.
Hence these are precisely the minimal probability metrics associated with the Prokhorov and $L_\infty$ metrics: $\widehat{d_p} = W_p, \phantom{1} \widehat{d_\infty} = W_\infty$.
For the other probability metrics defined in Examples \ref{vbn:PMs}, the following relations hold, see e.g. \cite{Z83} and \cite{R91}: $
\widehat{K_\lambda} = \rho_\lambda, \phantom{1} \widehat{d_i} = d_{TV}$.
\end{vbn}

%\section{The categorical setup}

\begin{lem}\label{lem:MinGauge}
Let $d_1$ and $d_2$ be probability metrics  and let $\mathcal{D}$ be a collection of probability metrics on $\mathcal{R}(X)$. Then
\begin{equation}
\sup_{d \in \mathcal{D}} \widehat{d} \:\:\leq\:\: \widehat{\sup_{d \in \mathcal{D}}}\label{eq:MinGauge1}
\end{equation}
and, for $\omega < \infty$ and $\epsilon > 0$,
\begin{equation}
d_1 \wedge \omega \leq d_2 + \epsilon \Rightarrow \widehat{d_1} \wedge \omega \leq \widehat{d_2} + \epsilon.\label{eq:MinGauge2}
\end{equation}
\end{lem}

\begin{proof}
Since $\sup_{d \in \mathcal{D}} \widehat{d} \:\:\leq\:\: \sup_{d \in \mathcal{D}}$
it follows that for $\xi, \eta, \xi^\prime, \eta^\prime \in \mathcal{R}(X)$ with $\mathcal{L}(\xi) = \mathcal{L}(\xi^\prime)$ and $\mathcal{L}(\eta) = \mathcal{L}(\eta^\prime)$ we have
\begin{equation*}
\sup_{d \in \mathcal{D}} \widehat{d}(\xi, \eta) = \sup_{d \in \mathcal{D}} \widehat{d}(\xi', \eta') \leq \sup_{d \in \mathcal{D}} d(\xi', \eta')\end{equation*}
which proves that
\begin{equation*}
\sup_{d \in \mathcal{D}} \widehat{d}(\xi, \eta) \:\:\leq\:\: \widehat{\sup_{d \in \mathcal{D}}}(\xi, \eta).\end{equation*}
Further (\ref{eq:MinGauge2}) follows at once from the definition.
\end{proof}

%Note that actually, (2) holds for arbitrary collections of probability metrics.

%Let $\mathcal{G}$ be a probability uniform gauge on $\mathcal{R}(X)$ and $\mathcal{D}$ a basis consisting of probability metrics. Then $\widehat{\mathcal{D}} = \left\{\widehat{d} \mid d \in \mathcal{D}\right\}$ is a basis for a unique probability uniform gauge $\widehat{\mathcal{G}}$ on $\mathcal{R}(X)$. Moreover, $\widehat{\mathcal{G}}$ is simple and does not depend on the particularly chosen basis $\mathcal{D}$.

%\begin{proof}
%This follows from Lemma \ref{lem:MinGauge}.
%\end{proof}

%\section{Categorical considerations}

%Theorem \ref{thm:MinGauge}, which allows us to take minimal probability uniform gauges, has the following consequence for the categories $\mathsf{Pol_{u}}$ and $\mathsf{sPol_{u}}$.

\begin{thm}
$\mathsf{sPol_{u}}$ is a reflective subcategory of $\mathsf{Pol_{u}}$ and $\mathsf{Pol_m}$ is a coreflective subcategory of $\mathsf{Pol_{u}}$.
 $$
\xy
\xymatrix{ 
\mathsf{Pol_u}   &  \mathsf{Pol_m} \ar[l]_c  \\
 \mathsf{sPol_u} \ar[u]^r & \mathsf{sPol_m} \ar[l]^c \ar[u]_r
}
\endxy
$$

\end{thm}

\begin{proof}

For the reflectivity of $\mathsf{sPol_u}$ in $\mathsf{Pol_u}$, let $ (X, \mathcal{G})$ be an object in $\mathsf{Pol_{u}}$ and let $\mathcal{D}$ be a basis for $\mathcal{G}$ consisting of probability metrics.

Then it follows from Lemma \ref{lem:MinGauge} that $\widehat{\mathcal{D}} = \left\{\widehat{d} \mid d \in \mathcal{D}\right\}$ is a basis for a unique simple probability uniform gauge $\widehat{\mathcal{G}}$ on $\mathcal{R}(X)$. Moreover, also from Lemma \ref{lem:MinGauge} it follows that $\widehat{\mathcal{G}}$ does not depend on the particular chosen basis $\mathcal{D}$. Hence  $(X,\widehat{\mathcal{G}})$ is a well-defined object in $\mathsf{sPol_{u}}$ only depending on $(X, \mathcal{G})$. 

The identity map
\begin{equation}
1_X : (X,\mathcal{G}) \rightarrow \left(X,\widehat{\mathcal{G}}\right)\label{eq:reflector}
\end{equation}
moreover defines a reflector. 

The fact that $\widehat{d} \leq d$ for each $d \in \mathcal{G}$ entails that (\ref{eq:reflector}) is a random contraction.
Further, let $(Y,\mathcal{H})$ be an object in $\mathsf{sPol_{u}}$, let $\mathcal{H}_0$ be a basis for $\mathcal{H}$ consisting of simple probability metrics, and let
\begin{equation*}
f : (X,\mathcal{G}) \rightarrow (Y,\mathcal{H})
\end{equation*}
be a random contraction. Then, for $e \in \mathcal{H}_0$, $\omega < \infty$, and $\epsilon > 0$, we find $d \in \mathcal{G}$ such that
\begin{equation*}
e \circ (\overline{f} \times \overline{f}) \wedge \omega \leq d + \epsilon,
\end{equation*}
which, $e$ being simple, by Lemma \ref{lem:MinGauge} leads to
\begin{equation*}
e \circ (\overline{f} \times \overline{f}) \wedge \omega \leq \widehat{d} + \epsilon. 
\end{equation*}
We conclude that
\begin{equation*}
\widehat{f} : \left(X,\widehat{\mathcal{G}}\right) \rightarrow (Y,\mathcal{H}),
\end{equation*}
where $\widehat{f}(x) = f(x)$, is the unique morphism between $\left(X,\widehat{\mathcal{G}}\right)$ and $(Y,\mathcal{H})$ for which $\widehat{f} \circ 1_X = f$. Since measurability is in all instances undisturbed, it follows that (\ref{eq:reflector}) is indeed a reflector.

For the coreflectivity of $\mathsf{Pol_m}$ in $\mathsf{Pol_u}$, first note that $\mathsf{Pol_m}$ is embedded as a subcategory of $\mathsf{Pol_u}$ in the usual way via the principal gauge generated by a unique metric (see \cite{L15}). Then again let $ (X, \mathcal{G})$ be an object in $\mathsf{Pol_{u}}$ and let $\mathcal{D}$ be a basis for $\mathcal{G}$ consisting of probability metrics and define $d_\mathcal{G} := \sup_{d \in \mathcal{D}} d$. It follows immediately from the definitions that $d_\mathcal{G}$ too is a probability metric and that its definition is independent from the chosen basis. Thus $(X, d_\mathcal{G})$ is a well-defined object in $\mathsf{Pol_m}$. It now follows from the general theory in \cite{L15} that 
$$1_X: (X, d_\mathcal{G}) \rightarrow (X, \mathcal{G})$$ indeed defines a coreflector.

The other reflectivity and coreflectivity claims are immediate consequences.
\end{proof}

% $$
%\xy
%\xymatrix{ 
%\mathsf{Pol_u}   &  \mathsf{Pol_m} \ar[l]_c  \\
% \mathsf{sPol_u} \ar[u]^r & \mathsf{sPol_m} \ar[l]^c \ar[u]_r
%}
%\endxy
%$$

The probability uniform gauge $\widehat{\mathcal{G}}$ is called the {\em minimal probability uniform gauge} associated with $\mathcal{G}$. Note that the minimal probability uniform gauge associated with a principal probability uniform gauge generated by a unique metric is the simple metric associated with that unique metric. Note also that the diagram in the foregoing theorem is not commutative.

\begin{vb}\label{vb:MinPG}
\begin{enumerate}
\item For the {\em Ky-Fan probability uniform gauge} as defined in Examples \ref{vbn:PGs} we find that $\widehat{\mathcal{G}_K}$ is the Prokhorov probability uniform gauge and that $d_{\mathcal{G}_K}$ is the indicator metric.
\end{enumerate}
\end{vb}

%\section{The minimal limit operator}

Theorem \ref{thm:LimOpMet} provides an explicit formula for the limit operator associated with a minimal probability metric. Its proof is based on Lemma \ref{lem:TechPaste}, which is a consequence of the well-known Kolmogorov Extension Theorem. 

\begin{thm}[Kolmogorov Extension Theorem]\label{thm:KolEx}
Let $I$ be a non-empty set, $\{X_i \mid i \in I\}$ a collection of Polish spaces, and $\pi_J \in \mathcal{P}(\Pi_{j \in J} X_j)$ for each finite set $J \subset I$. Then there exists $\pi \in \mathcal{P}(\Pi_{i \in I} X_i)$ such that $\pi^J = \pi_J$ for each finite set $J \subset I$ if and only if the collection $\{\pi_{J} \mid J \subset I \textrm{ finite}\}$ is consistent in the sense that $\pi_{J_1}^{J_1 \cap J_2} = \pi_{J_2}^{J_1 \cap J_2}$ for all finite sets $J_1, J_2 \subset I$ with $J_1 \cap J_2 \neq \emptyset$. 
\end{thm}

\begin{lem}\label{lem:TechPaste}
Let $\{X_n \mid n \in \mathbb{N}\}$ be a collection of Polish spaces and let $\{\pi_{\{0,n\}} \in \mathcal{P}(X_{0} \times X_{n}) \mid n \in \mathbb{N}_0\}$ be a consistent collection of measures in the sense that $\pi_{\{0,n\}}^{\{0\}} = \pi_{\{0,m\}}^{\{0\}}$ for all $n, m \in \mathbb{N}_0$. Then there exists a measure $\pi \in \mathcal{P}(\Pi_{n \in \mathbb{N}} X_n)$ such that $\pi^{\{0,n\}} = \pi_{\{0,n\}}$ for each $n \in \mathbb{N}_0$.
\end{lem}

\begin{proof}
Inductively applying Lemma \ref{lem:PastingLemma} shows that for each $n \in \mathbb{N}_0$ there exists a probability measure $\pi_{\{0,\ldots,n\}} \in \mathcal{P}(\Pi_{k = 1}^n X_k)$ with the property that $\pi_{\{0,\ldots,n\}}^{J_m} = \pi_{J_m}$ for all $1 \leq m \leq n$. Furthermore, for a finite set $F \subset \mathbb{N}$, define $\pi_F \in \mathcal{P}(\Pi_{k \in F} X_k)$ by putting $\pi_F = \pi_{\{0,\ldots,\max F\}}^{F}$. Then $\{\pi_{F} \mid F \subset \mathbb{N} \textrm{ finite} \}$ is a consistent collection, and Theorem \ref{thm:KolEx} provides the existence of the desired probability measure. 
\end{proof}

%\begin{rem}
%Lemma \ref{lem:TechPaste} is a specific instance of what is called the Marginal Problem. For a general treatment of this problem, we refer the reader to \cite{DKS91}.
%\end{rem}

\begin{thm}\label{thm:LimOpMet}
Let $X$ be a Polish space and $d$ a probability metric on $\mathcal{R}(X)$. Then, for $\xi \in \mathcal{R}(X)$ and $(\xi_{n})_n$ in $\mathcal{R}(X)$,
\begin{equation}
\lambda(\xi_n \rightarrow \xi) = \limsup_{n \rightarrow \infty} \widehat{d}\left(\xi, \xi_n\right) = \min \limsup_{n \rightarrow \infty} d\left(\xi^\prime,\xi_n^\prime\right),\label{eq:PIneq}
\end{equation}
the minimum taken over all $\xi^\prime \in \mathcal{R}(X)$ and all sequences $(\xi^\prime_n)_n$ in $\mathcal{R}(X)$ such that $\mathcal{L}(\xi^\prime) = \mathcal{L}(\xi)$ and $\mathcal{L}(\xi_n^\prime) = \mathcal{L}(\xi_n)$ for all $n$. 
\end{thm}

\begin{proof}
For $n \in \mathbb{N}_0$, choose $\xi^{(n)}$ and $\xi^0_{n}$ in $\mathcal{R}(X)$ such that $\mathcal{L}(\xi^{(n)}) = \mathcal{L}(\xi)$, $\mathcal{L}(\xi^0_n) = \mathcal{L}(\xi_n)$, and 
\begin{equation}
d(\xi^{(n)},\xi^0_n) \leq \widehat{d}(\xi,\xi_n) + 1/n.\label{eq:techin1}
\end{equation} 
Put $X_0 = X$, and, for $n \in \mathbb{N}_0$, $X_n = X$ and $\pi_{\{0,n\}} = \mathcal{L}(\xi^{(n)},\xi^0_n) \in \mathcal{P}(X_0 \times X_n)$. Then, by Lemma \ref{lem:TechPaste}, there exists a measure $\pi \in \mathcal{P}(\Pi_{n \in \mathbb{N}} X_n)$ such that $\pi^{\{0,n\}} = \mathcal{L}(\xi^{(n)},\xi_n^0)$ for all $n \in \mathbb{N}_0$. Furthermore, our assumption on $\Omega$ provides the existence of random variables $\xi^{\prime}$ and $(\xi^\prime_{n})_n$ in $\mathcal{R}(X)$ such that
\begin{equation}
\mathcal{L}(\xi^\prime,\xi^\prime_{n}) = \pi^{\{0,n\}} = \mathcal{L}(\xi^{(n)},\xi^0_n)\label{eq:techin2}
\end{equation}
for all $n \in \mathbb{N}_0$. From (\ref{eq:techin1}) and (\ref{eq:techin2}) we now infer that, for $n \in \mathbb{N}_0$,
\begin{equation}
d(\xi^\prime,\xi^\prime_{n}) \leq \widehat{d}(\xi,\xi_n) + 1/n,
\end{equation}
whence we conclude that the right-hand side of (\ref{eq:PIneq}) is dominated by the left-hand side of (\ref{eq:PIneq}). The converse inequality follows by definition of $\widehat{d}$.

\end{proof}

\begin{gev}\label{gev:GenConv}
Let $X$ be a Polish space and $d$ a probability metric on $\mathcal{R}(X)$. Then, for $\xi \in \mathcal{R}(X)$ and $(\xi_{n})_n$ in $\mathcal{R}(X)$,
\begin{equation*}
\lim_n \widehat{d}(\xi,\xi_n) = 0
\end{equation*}
if and only if there exist  $\xi^\prime \in \mathcal{R}(X)$ and $(\xi_n^\prime)_n$ in $\mathcal{R}(X)$ such that 
\begin{equation*}
\mathcal{L}(\xi^\prime) = \mathcal{L}(\xi), \phantom{1} \forall n : \mathcal{L}(\xi_n^\prime) = \mathcal{L}(\xi_n), \textrm{ and } \lim_n d(\xi^\prime,\xi^\prime_n) = 0.
\end{equation*}
\end{gev}

The above corollary unifies and generalizes various separate results in the literature.

\begin{enumerate}

\item It is proved for the the Ky-Fan metric $K_1$ making use of the Skorokhod-Dudley Theorem in \cite{Dud02}.

\item It is proved for the $L^p$-metric $d_p$ in \cite{V03}.

\item It is proved  for the $L^\infty$-metric $d_\infty$ in \cite{D02}.

\item Finally it is again proved for the total variation metric $d_{TV}$ in \cite{JR97}.

\end{enumerate}

All these results are obtained with different proofs. The general proof of Theorem \ref{thm:LimOpMet} presented here, is inspired by \cite{D02}.

The following result shows that Theorem \ref{thm:LimOpMet} is partially extendable to the limit operator associated with the approach structure underlying a minimal probability uniform gauge. 

\begin{thm}\label{thm:MinLimOp}
Let $(X,\mathcal{G})$ be an object in $\mathsf{Pol_u}$ and $\lambda_{\mathcal{G}}$ (respectively $\lambda_{\widehat{\mathcal{G}}}$) the limit operator associated with the approach structure underlying $\mathcal{G}$ (respectively $\widehat{\mathcal{G}}$). Then, for $\xi \in \mathcal{R}(X)$ and $(\xi_{n})_n$ in $\mathcal{R}(X)$,
\begin{equation}
\lambda_{\widehat{\mathcal{G}}}\left(\xi_n \rightarrow \xi\right) \leq \inf \lambda_{\mathcal{G}}\left(\xi_n^\prime \rightarrow \xi^\prime\right),\label{eq:MinLimOp}
\end{equation}
the infimum taken over all $\xi^\prime \in \mathcal{R}(X)$ and all sequences $(\xi^\prime_n)_n$ in $\mathcal{R}(X)$ such that $\mathcal{L}(\xi^\prime) = \mathcal{L}(\xi)$ and $\mathcal{L}(\xi_n^\prime) = \mathcal{L}(\xi_n)$ for all $n$.
\end{thm}

\begin{proof}
By definition of $\widehat{d}$, we have, for $n \in \mathbb{N}$, 
\begin{equation*}
\widehat{d}(\xi,\xi_n) \leq d(\xi^\prime,\xi_n^\prime)
\end{equation*}
for all $\xi^\prime, \xi_n^\prime \in \mathcal{R}(X)$ with $\mathcal{L}(\xi^\prime) = \mathcal{L}(\xi)$ and $\mathcal{L}(\xi_n^\prime) = \mathcal{L}(\xi_n)$. This easily gives (\ref{eq:MinLimOp}).
\end{proof}

Whether or not the inequality in (\ref{eq:MinLimOp}) is strict in general, remains an open problem.

Let $(X,d)$ be a complete and separable metric space, $\lambda_w$ the limit operator associated with the weak approach structure on $\mathcal{R}(X)$, and $\lambda_{\mathbb{P}}$ the limit operator associated with the approach structure of convergence in probability on $\mathcal{R}(X)$, see e.g. \cite{L15}. In \cite{L15}, p. 312, it is shown that, for $\xi \in \mathcal{R}(X)$ and $(\xi_{n})_n$ in $\mathcal{R}(X)$, the inequality 
\begin{equation*}
\lambda_w(\xi_n \rightarrow \xi) \leq \lambda_{\mathbb{P}}(\xi_n \rightarrow \xi) 
\end{equation*}
always holds. Again, we derive this result as a corollary from Theorem \ref{thm:MinLimOp}. 

\begin{gev}\label{last}
For $\xi \in \mathcal{R}(X)$ and $(\xi_n)_n$ in $\mathcal{R}(X)$, 
\begin{equation*}
\lambda_w(\xi_n \rightarrow \xi) \leq \inf \lambda_{\mathbb{P}}\left(\xi^\prime_n \rightarrow \xi^\prime\right), 
\end{equation*}
the infimum taken over all $\xi^\prime \in \mathcal{R}(X)$ and all $(\xi^\prime_n)_n$ in $\mathcal{R}(X)$ such that $\mathcal{L}(\xi^\prime) = \mathcal{L}(\xi)$ and $\mathcal{L}(\xi_n^\prime) = \mathcal{L}(\xi_n)$ for all $n$.
\end{gev}

\begin{proof}
From Examples \ref{vbn:PGs} we know that $\mathcal{A}_{\mathbb{P}}$ is the underlying approach structure of $\mathcal{G}_K$, and $\mathcal{A}_{w}$ the underlying approach structure of $\mathcal{G}_P$. Furthermore, in Example \ref{vb:MinPG} we have established that $\widehat{\mathcal{G}_K} = \mathcal{G}_P$. Now it suffices to apply Theorem \ref{thm:MinLimOp} in the case where $\mathcal{G} = \mathcal{G}_K$.
\end{proof}

\end{document}